\documentclass[12pt,reqno]{article}

\usepackage[usenames]{color}
\usepackage{amssymb}
\usepackage{graphicx}
\usepackage{amscd}

\usepackage[colorlinks=true,
linkcolor=webgreen,
filecolor=webbrown,
citecolor=webgreen]{hyperref}

\definecolor{webgreen}{rgb}{0,.5,0}
\definecolor{webbrown}{rgb}{.6,0,0}

\usepackage{color}
\usepackage{fullpage}
\usepackage{float}

\usepackage{graphics,amsmath,amssymb}
\usepackage{amsthm}
\usepackage{amsfonts}
\usepackage{latexsym}
\usepackage{epsf}

\usepackage{verbatim}
\usepackage{enumerate}

\usepackage{tikz}
\usepackage[blocks]{authblk}

\usepackage{algpseudocode}

\theoremstyle{plain}
\newtheorem{theorem}{Theorem}
\newtheorem{corollary}[theorem]{Corollary}
\newtheorem{lemma}[theorem]{Lemma}
\newtheorem{proposition}[theorem]{Proposition}
\newtheorem{conjecture}[theorem]{Conjecture}

\theoremstyle{definition}

\theoremstyle{remark}

\begin{document}

\title{Fibonacci Partial Sums Tricks}
\author{Nikhil Byrapuram}
\author{Adam Ge}
\author{Selena Ge}
\author{Sylvia Zia Lee}
\author{Rajarshi Mandal}
\author{Gordon Redwine}
\author{Soham Samanta}
\author{Daniel Wu}
\author{Danyang Xu}
\author{Ray Zhao}
\affil{PRIMES STEP}
\author{Tanya Khovanova}
\affil{MIT}

\maketitle

\begin{abstract}
The following magic trick is at the center of this paper. While the audience writes the first ten terms of a Fibonacci-like sequence (the sequence following the same recursion as the Fibonacci sequence), the magician calculates the sum of these ten terms very fast by multiplying the 7th term by 11. This trick is based on the divisibility properties of partial sums of Fibonacci-like sequences. We find the maximum Fibonacci number that divides the sum of the Fibonacci numbers 1 through $n$. We discuss the generalization of the trick for other second-order recurrences. We show that a similar trick exists for Pell-like sequences and does not exist for Jacobhstal-like sequences.
\end{abstract}

\section{Introduction}

\subsection{The trick}

Everyone loves Fibonacci numbers. Moreover, everyone loves magic. This paper is about a magic trick related to Fibonacci numbers \cite{FNT}.

Consider the sequence of Fibonacci numbers $F_n$ defined as: $F_0 = 0$, $F_1 = 1$, and for $n > 1$, we have $F_n = F_{n-1} + F_{n-2}$. We call an integer sequence $S_n$ \textit{Fibonacci-like} if it follows the same rule as the Fibonacci sequence: $S_n = S_{n-1} + S_{n-2}$.
	
\textbf{The trick.} Ask your friend for two numbers. Create a Fibonacci-like sequence that starts with these two numbers and is exactly ten terms long. When you reach the tenth number, tell your friend you can total all ten numbers in your head!

\textbf{Explanation for the trick.} Suppose the first two terms of the Fibonacci-like sequence are $a$ and $b$. Then, one can check that the seventh term is $5a+8b$, and the sum of the first 10 terms is $55a+88b$. Thus, the total is always the seventh number times 11. As a bonus, multiplying by 11 in one's head is easy.

After we learned the trick, we noticed some other patterns. For example, if you add six consecutive terms of a Fibonacci-like sequence, you always get the fifth term times 4. Even simpler, if you add three terms, you get the third term times 2. And, even more trivially, when you add one term, you get the first term times 1. We summarize this in Table~\ref{tab:smalltrick}.

\begin{table}[ht!]
\centering
\begin{tabular}{|c|c|c|c|c|}
\hline
Number of elements $n$ in the sum 	& the sum & index $m$ &$m$-th term & multiplier $z$ \\ \hline
$1$	& $a$		& $1$	& $a$		& $1$ \\ \hline
$3$	& $2a+2b$	& $3$	& $a+b$	& $2$ \\ \hline
$6$	& $8a+12b$	& $5$	& $2a+3b$	& $4$ \\ \hline
$10$	& $55a+88b$	& $7$	& $5a+8b$	& $11$ \\
\hline
\end{tabular}
\caption{Sum of the first $n$ terms in a Fibonacci-like sequence is equal the $m$-th term times $z$.}
\label{tab:smalltrick}
\end{table}

The first column looks like triangular numbers. The third column looks like odd numbers. Do we have a pattern? Read on to find out!

\subsection{The statement of results and the road map}

We start with Fibonacci numbers in Section~\ref{sec:fib}. We consider the partial sums $S_n^F$ of the first $n$ terms ($F_1$ through $F_n$) of the Fibonacci sequence. We are looking for the largest Fibonacci number that divides $S_n^F$. We prove in Corollary~\ref{cor:psdiv} that when we divide partial sums $S_n^F$ by a Fibonacci number $F_m$, we get a Lucas number under certain conditions on $m$ and $n$. Our maximal $m$ satisfies these conditions. Theorem~\ref{thm:Fibmax} describes the index of the largest Fibonacci number that divides the sum of the first $n$ Fibonacci terms. To prove the theorem, we study in how many ways a number can be decomposed as a product of a Fibonacci number and a Lucas number. It appears that such a product is unique up to a small number of exceptions. This is covered by Proposition~\ref{prop:FLIdentities}. 

In Section~\ref{sec:flike}, we generalize the discussion to any Fibonacci-like sequence. Theorem~\ref{thm:Fib-like} states that the trick can be extended from Fibonaccis to any Fibonacci-like sequence for $n=3$ or $n \equiv 2 \pmod{4}$.

We discuss the maximality of Lucas sequences in Section~\ref{sec:lucas}.

We transition to general Lucas sequences in Section~\ref{sec:gen}. We explain when the trick works for all the sequences with the same recurrence. We also discuss particular examples. For the recurrence $x_n = -x_{n-1} + x_{n-2}$, the trick works for $n = 2k+2$. For the recurrence $x_n = x_{n-1} - x_{n-2}$, the trick works for $n \equiv 1,\ 3,\ 5,\ 6 \pmod{6}$. For the recurrence $x_n = 3x_{n-1} - x_{n-2}$, the trick works for odd $n$.

In Section~\ref{sec:jacob}, we show that the trick for Jacobsthal-like sequences does not exist. However, the trick exists for Pell-like sequences, as we show in Section~\ref{sec:pell}. We finish with Tribonacci and $n$-nacci sequences in Section~\ref{sec:trib}.

\section{Fibonacci Numbers}
\label{sec:fib}

\subsection{Motivation}

For the trick to work for any Fibonacci-like sequence, it needs to also work for the Fibonacci sequence itself. Thus, we start with the Fibonacci sequence. Define $S^F_n$ to be the sum of the first $n$ Fibonacci numbers, namely $F_1$ through $F_n$. The formula for this sum is well-known:
\[S^F_n = \sum_{i=1}^n F_i = F_{n+2} - 1.\]

As we showed before, $S^F_{10} = 11F_7$. To perform that trick, it is useful that the index 7 is relatively large. This way, we need to multiply by a relatively small number, making it possible to do it in our heads. Thus, we need to find the largest Fibonacci number that divides the sum of the first $n$ Fibonacci numbers.

\subsection{Preliminaries}

Lucas numbers $L_n$ play a special role in this section. They start as $L_0 = 2$ and $L_1 = 1$. The sequence follows the same rule as the Fibonaccis: for $n > 1$, we have $L_n = L_{n-1} + L_{n-2}$.

Our main question is: given $n$, what is the largest $m$ such that the sum of the first $n$ Fibonaccis is divisible by $F_m$? Let $m$ be the largest index of a Fibonacci number that divides $S^F_n$, and let us denote the factor by $z$. We have $S^F_n = zF_m$.

We wrote a program to calculate the values $m$ and $z$ and realized that the multiplier $z$ is always a Lucas number: $z = L_i$. Values $m$, $z$, and $i$ should be viewed as functions of $n$. When we need to emphasize this, we will use $m_n$, $z_n$, and $i_n$. Table~\ref{tab:trickAnyFib} shows these values for small $n$.

\begin{table}[ht!]
\begin{center}
\begin{tabular}{|c|c|c|c|}
\hline
Number of elements $n$ in the sum 	& index $m$	& multiplier $z$ & index $i$ of $z$\\ \hline
1	& 2	& 1 	& 1 \\ \hline
2	& 3	& 1 	& 1 \\ \hline
3	& 3	& 2 	& 0 \\ \hline
4	& 2	& 7 	& 4 \\ \hline
5	& 4	& 4 	& 3 \\ \hline
6	& 5	& 4 	& 3 \\ \hline
7	& 4	& 11 	& 5 \\ \hline
8	& 4	& 18 	& 6 \\ \hline
9	& 6	& 11 	& 5 \\ \hline
10	& 7	& 11 	& 5 \\ \hline
11	& 6	& 29 	& 7 \\ \hline
12	& 6	& 47 	& 8 \\ \hline
13	& 8	& 29 	& 7 \\ \hline
14	& 9	& 29 	& 7 \\ \hline
15	& 8	& 76 	& 9 \\ \hline
16	& 8	& 123 	& 10 \\ \hline
17	& 10	& 76 	& 9 \\ \hline
18	& 11	& 76 	& 9 \\ \hline
\end{tabular}
\end{center}
\caption{Sum of the first $n$ Fibonacci numbers equals the $m$-th term times $z$, where $z = L_i$.} 
\label{tab:trickAnyFib}
\end{table}

We see a pattern that starts from $n > 3$, where we have $m_{n+4} = m_n +1$, and $i_{n+4} = m_n +2$, which we will prove in Theorem~\ref{thm:Fibmax}. We have a sequence $m_n$, which is a new sequence A372048 in the OEIS \cite{OEIS}:
\[2, 3, 2, 2, 4, 5, 4, 4, 6, 7, 6, 6, 8, 9, 8, 8, 10, 11, 10, 10, 12, 13, 12, 12, \ldots,\]
and sequence $i_n$, which is also a new sequence A372049:
\[1, 1, 0, 4, 3, 3, 5, 6, 5, 5, 7, 8, 7, 7, 9, 10, 9, 9, 11, 12, 11, 11, 13, 14, \ldots.\]

We would like to introduce some equalities that will be useful later. We start with a famous equality 
\begin{equation}
\label{eq:n|n}
F_{2n} = F_nL_n.
\end{equation}
It is also known \cite{TG} that the Fibonacci neighbors can be factored into a product of a Fibonacci number and a Lucas number:
\begin{align*}
F_{2n} + (-1)^n & = F_{n-1} L_{n+1}\\
F_{2n} - (-1)^n & = F_{n+1} L_{n-1}\\
F_{2n+1} + (-1)^n & = F_{n+1} L_{n}\\
F_{2n+1} - (-1)^n & = F_{n} L_{n+1}.
\end{align*}

For our trick, we are interested in partial sums, so we will use the values of $n$ such that we can get partial sums of the Fibonacci sequence on the left side in the equations above. We rewrite the equations above, plugging in $2k+1$, $2k+2$, $2k+1$, and $2k+2$ for $n$, correspondingly. We get,
\begin{align}
F_{4k+2} - 1 & = F_{2k} L_{2k+2} = S^F_{4k}
\label{eq:2k+2|2k+4}\\
F_{4k+4} - 1 & = F_{2k+3} L_{2k+1} = S^F_{4k+2}
\label{eq:2k+3|2k+1}\\
F_{4k+3} - 1 & = F_{2k+2} L_{2k+1} = S^F_{4k+1}
\label{eq:2k+2|2k+1}\\
F_{4k+5} - 1 & = F_{2k+2} L_{2k+3} = S^F_{4k+3}
\label{eq:2k+2|2k+3}.
\end{align}

We see that sums of the first $n$ Fibonacci numbers are divisible by a Fibonacci number with an index close to $\frac{n}{2}$. But is such an index the largest possible? Before discussing it, we need to analyze the products of Fibonacci and Lucas numbers.

\subsection{Products of Fibonacci and Lucas numbers}

We see that the sums of Fibonacci numbers are products of Fibonacci and Lucas numbers. But how many ways can a number be a product of a Fibonacci number and a Lucas number? We would like to know all values $a$, $b$, $c$, $d$ such that $F_a L_b = F_c L_d$, assuming that indices are nonnegative.

\begin{proposition}
\label{prop:FLIdentities}
If $F_a L_b = F_c L_d$, and the indices are nonnegative, then one of the following cases is true.
\begin{enumerate}
\item $a = c$ and $b = d$, corresponding to the equality $F_a L_b = F_a L_b$.
\item $a=c=0$, corresponding to the equality $F_0 L_b = F_0 L_d = 0$.
\item $a=1$, $c=2$, and $b=d$, corresponding to the equality $F_1 L_b = F_2 L_b = L_b$ as $F_1 = F_2 = 1$. We have a similar case when $a=2$, $c=1$, and $b=d$.
\item $a=b=k$, $c=2k$, and $d=1$, corresponding to a famous equation $F_k L_k = F_{2k} L_1 = F_{2k}$.
\item ($a$, $b$, $c$, $d$) is one of (1, 3, 3, 0), (2, 3, 3, 0), (3, 1, 1, 0), (3, 1, 2, 0), and (3, 2, 4, 0), corresponding to the few exceptions when the product is 2, 4, or 6. The exceptions are
$F_1 L_0 = F_2 L_0 = F_3 L_ 1 = 2$, $F_1 L_3 = F_2 L_3 = F_3 L_0 = 4$, and $F_3 L_2 = F_4 L_0 = 6$.
\end{enumerate}
\end{proposition}

\begin{proof}
To prove that these are all the cases, we use the following equality $F_{a}L_{b} = F_{a + b} + (-1)^bF_{a - b}$, which allows us to replace $F_a L_b = F_c L_d$ with
\begin{equation}
\label{eq:product2sum}
F_{a + b} + (-1)^bF_{a - b} = F_{c + d} + (-1)^dF_{c - d}.
\end{equation}

Without loss of generality, we assume that $a + b \geq c + d$. We consider four cases, denoted by 0a, 0b, 0d, and 4, where we assume $a=0$, $b=0$, $d=0$, and all four indices are positive, respectively.

\textbf{Case 0a.} If one of $a$ or $c$ is zero, the other is zero, too, and we get Case 2 above. From now on, we assume that $a$ and $c$ are positive.

\textbf{Case 0b.} Suppose $b =0$, then $2F_{a} = F_{c + d} + (-1)^dF_{c - d}$. And, recall, we assumed $a \ge c+d$. This could only happen when $c = a$ and $d = 0$. Thus, we get Case 1 above. From now on, we assume that $b$ is positive.

\textbf{Case 0d.} Assume that $d = 0$, then we get $F_{a + b} + (-1)^bF_{a - b} = 2F_{c}$. As $a$ and $b$ are positive, we have $|a-b| \leq a+b -2$. That means, $c < a+b$.

We consider three subcases: $c= a+b-1$, $c = a+b-2$ and $c < a+b-2$.

\textbf{Subcase 1.} Suppose $c= a+b-1$, then we have $F_{c+1} + (-1)^bF_{a - b} = 2F_{c}$, which we can simplify to $(-1)^bF_{a - b} = F_c-F_{c-1} = F_{c-2}$. By our assumption, $c$ has a different parity from $a \pm b$. Thus, $|c-2| \ne |a-b|$. It follows that $F_{c-2} = 1$, implying that $c - 2$ and $|a-b|$ are in the set $\{-1, 1, 2\}$.

Suppose $c-2=-1$ and $|a-b|=2$. Then $a+b=c+1=2$, but since $|a-b|=2$, one of $a$ and $b$ equals zero, corresponding to Case 1 or Case 2 of the proposition. 

Suppose $c-2=1$ and $|a-b|=2$. Then $a+b=c+1=4$, and since $|a-b|=2$, $a$ and $b$ are equal to 1 and 3 in some order. Testing both cases, we find that the only valid solution to ($a$, $b$, $c$, $d$) is (1, 3, 3, 0), which is one of the exceptions in Case 5 of the proposition. 

Suppose $c-2=2$ and $|a-b|=1$. Then $a+b=c+1=5$, and since $|a-b|=1$, $a$ and $b$ are equal to $2$ and $3$ in some order. Testing both cases, we find that the only valid solution is (3, 2, 4, 0), which is one of the exceptions in Case 5 of the proposition.

\textbf{Subcase 2.} We now suppose that $c= a+b-2$. We get $F_{c+2}\pm F_{|a-b|}=2F_c$, which we can simplify to $-F_{|a-b|}=F_c+(F_c-F_{c+2})=F_c-F_{c+1}=-F_{c-1}$. Except for when $c-1$ equals 1 or 2, $|a-b|$ must equal $c-1$. But $a+b=c+2$, and as the sum and absolute difference between $a$ and $b$ have different parities, $a$ and $b$ are not integers, which is a contradiction. Now we consider the case when $c-1$ equals 1 or 2 in some order. If $c-1=1$ and $|a-b|=2$, then $a+b=c+2=4$. Thus, $a$ and $b$ are equal to 1 and 3 in some order. Testing both cases, we find that the only valid solution is (3, 1, 2, 0), which is one of the exceptions in Case 5 of the proposition. If $c-1=2$ and $|a-b|=1$, then $a+b=c+2=5$. Thus, $a$ and $b$ are equal to 2 and 3 in some order. Testing both cases, we find that the only valid solution is (2, 3, 3, 0), which is one of the exceptions in Case 5 of the proposition.

\textbf{Subcase 3.} The value of $2F_{c} - (-1)^bF_{a - b}$ is not greater than $2F_{a+b -3} + F_{a +b - 2} = F_{a+b -3} + F_{a +b -1}$. This could potentially equal $F_{a + b}$ only if $a+b-3 = 1$. Thus, we have $c = 1$ and $a+b =4$, which corresponds to the equation $F_{4} + (-1)^bF_{a - b} = 2F_{1}$, or equivalently $(-1)^bF_{a - b} = -1$. The solution (3, 1, 1, 0) follows, which is one of the exceptions in Case 5.

\textbf{Case 4.} 
From now on, we assume that all indices are positive. Then
\[|F_{a - b}| \leq F_{a + b-2} \quad \textrm{ and } \quad |F_{c - d}| \leq F_{c + d-2} .\]
Therefore,
\[F_{a + b} + (-1)^bF_{a - b} \ge F_{a + b} - F_{a +b -2} = F_{a +b -1}\]
and
\[F_{c + d} + (-1)^dF_{c - d} \le F_{c + d} + F_{c + d-2} \le F_{c + d + 1}.\]
Moreover, the last sign is an equality only when $c+d-2 = 1$, meaning that either $c = 1$ and $d=2$, or $c = 2$, and $d=1$. We can manually check that these values correspond to Case 3 in the proposition.
From now on, we assume that $F_{c + d} + F_{c + d-2} < F_{c + d + 1}$ from which it follows that
\[F_{a +b -1} \le F_{a + b} + (-1)^bF_{a - b} = F_{c + d} + (-1)^dF_{c - d} < F_{c + d + 1},\]
implying that $a+b-1 < c+d+1$, or, equivalently $a+b < c+d+2$. Combining this with our initial assumption that $c+d \le a+b$, we conclude that $c + d \leq a + b \leq c + d + 1$. We divide the rest in two subcases.

\textbf{Subcase 1.} Suppose $a+b=c+d$. We can simplify Eq.~\ref{eq:product2sum} by canceling $F_{a+b}=F_{c+d}$:
\[(-1)^bF_{a-b}=(-1)^dF_{c-d}.\]
It follows that $F_{|a-b|}=F_{|c-d|}$. Moreover, as $a+b$ and $c+d$ have the same parity, so $|a-b|$ and $|c-d|$ have the same parity. Thus, $|a-b| = |c-d|$.

Since the absolute difference is the same for pairs $(a,b)$ and $(c,d)$ and their sums are equal, there are two cases: $a=c$ and $b=d$ or $a=d$ and $b=c$. The former corresponds to Case 1 of the proposition. Consider the latter case, then $(-1)^dF_{c-d} = (-1)^aF_{b-a}$. On the other hand, we are given that $(-1)^bF_{a-b}=(-1)^dF_{c-d}$. It follows that $(-1)^bF_{a-b}=(-1)^aF_{b-a}$. By expanding, we get
\[(-1)^bF_{a-b}=(-1)^aF_{b-a} = (-1)^a(-1)^{b-a+1}F_{a-b} = (-1)^{b+1}F_{a-b}.\]
Therefore, $a-b = 0$, but then $a=c = b=d$, so this corresponds to Case 1 of the proposition too.

\textbf{Subcase 2.} Suppose $a+b=c+d+1$. Plugging this into Eq.~\ref{eq:product2sum}, we get:
\[F_{a+b}+(-1)^bF_{a-b} = F_{a+b-1}+(-1)^dF_{c-d}.\]
This can be rearranged into:
\[F_{a+b}-F_{a+b-1} = F_{a+b-2} = (-1)^dF_{c-d}-(-1)^bF_{a-b}.\]
We also know that $|a-b| \le a+b-2$ and has the same parity as $a+b$.

Suppose $|a-b| = a + b -2$. Then, either $a$ or $b$ has to be 1. Plugging in $a=1$ into the equation above we get $F_{b-1} = (-1)^dF_{c-d}-(-1)^bF_{1-b} = (-1)^dF_{c-d} + F_{b-1}$. Similarly, plugging $b=1$ into the equation above, we get $F_{a-1} = (-1)^dF_{c-d} + F_{a-1}$. In both cases, we get $F_{c-d} = 0$, which implies that $c=d$ and $a = 2c$, corresponding to Case 4.

Now we assume that $|a-b| < a+b-2$, we also have that $|c-d| \le c+d-2 < a+b -2$. Also, $a-b$ and $c-d$ have different parities. It follows that $|a-b| = a+b-4$, and $|c-d| = a+b-3 = c+d -2$. It follows that one of $a,b$ is 2, one of $c,d$ is 1, and the other two values equal each other. We can check the four cases for values of $(a,b,c,d)$. The values $(2,b,1,b)$ correspond to Case 3 of the proposition. The values $(a,2,1,a)$ mean $F_a = L_a$, which never happens. The values $(2,b,b,1)$ mean $L_b = 2F_b$, which never happens. The values $(a,2,a,1)$ mean $F_a = 2F_a$, which does not happen for positive $a$.
\end{proof}

\subsection{Partial sums divisibility}

We noticed that when partial sums divide a Fibonacci number, we often get a Lucas number as a result. It is not actually surprising.

Consider partial sums $S_n^F$ of the first $n$ Fibonacci numbers that equal to $F_{n+2} - 1$, which is close to $\frac{\phi^{n+2}}{\sqrt{5}}$. Number $F_m$ is close to $\frac{\phi^m}{\sqrt{5}}$. Therefore, the ratio $\frac{S_n^F}{F_m}$ is close to $\phi^{n-m+2}$, which itself is close to $L_{n-m+2}$. The following lemma estimates when the ratio $\frac{S_n^F}{F_m}$ is not more than 1 from a Lucas number.

\begin{lemma}
\label{lemma:estimate}
If $n-3m\leq-4$ and $n\geq11$ and $n - m \geq 0$, then $\left|\frac{F_{n+2}-1}{F_m}-L_{n-m+2}\right| < 1$.
\end{lemma}

\begin{proof}
Using Binet's formulas, we get that 
\[\frac{F_{n+2}-1}{F_m}-L_{n-m+2}=\frac{\frac{\varphi^{n+2}-(\frac{-1}{\varphi})^{n+2}}{\sqrt{5}}-1}{\frac{\varphi^m-(\frac{-1}{\varphi})^m}{\sqrt{5}}}-\left(\varphi^{n-m+2}+\left(\frac{-1}{\varphi}\right)^{n-m+2}\right)\]
\[=\frac{\varphi^{n+2}-(\frac{-1}{\varphi})^{n+2}-\sqrt{5}}{\varphi^m-(\frac{-1}{\varphi})^m}-\varphi^{n-m+2}-\left(\frac{-1}{\varphi}\right)^{n-m+2}\]
\[=\frac{\varphi^{n+2}-(\frac{-1}{\varphi})^{n+2}-\sqrt{5}-\varphi^{n+2}+(\frac{-1}{\varphi})^m\varphi^{n-m+2}}{\varphi^m-(\frac{-1}{\varphi})^m}-\left(\frac{-1}{\varphi}\right)^{n-m+2}\]
\[=\frac{-(\frac{-1}{\varphi})^{n+2}-\sqrt{5}+(-1)^m\varphi^{n-2m+2}}{\varphi^m-(\frac{-1}{\varphi})^m}-\left(\frac{-1}{\varphi}\right)^{n-m+2}\]
\[=\frac{-(\frac{-1}{\varphi})^{n+2}-\sqrt{5}+\varphi^{n-4m+2}}{\varphi^m-(\frac{-1}{\varphi})^m}+(-1)^m\varphi^{n-3m+2}-\left(\frac{-1}{\varphi}\right)^{n-m+2}\]
\[=\frac{-(\frac{-1}{\varphi})^{n+2}}{\varphi^m-(\frac{-1}{\varphi})^m} - \frac{\sqrt{5}}{\varphi^m-(\frac{-1}{\varphi})^m} + \frac{\varphi^{n-4m+2}}{\varphi^m-(\frac{-1}{\varphi})^m}+ (-1)^m\varphi^{n-3m+2}- \left(\frac{-1}{\varphi}\right)^{n-m+2}.\]

Therefore,
\[\left|\frac{F_{n+2}-1}{F_m}-L_{n-m+2}\right| \leq \frac{{\varphi}^{-n-2}}{\varphi^m-(\frac{-1}{\varphi})^m} + \frac{\sqrt{5}}{\varphi^m-(\frac{-1}{\varphi})^m} + \frac{\varphi^{n-4m+2}}{\varphi^m-(\frac{-1}{\varphi})^m} + \varphi^{n-3m+2} + \varphi^{m-n-2}.\]

We now recall two conditions: $n-3m\leq-4$ and $n\geq11$, from which we get $\frac{n+4}{3}\leq m$, and $m\geq 5$. We estimate the terms above one by one.

For $m \geq 5$, we have $\frac{1}{\varphi^m-(\frac{-1}{\varphi})^m} < 0.09$. For $n \geq 11$, we have $\varphi^{-n-2} < 0.002$. Moreover, $\sqrt{5} < 2.24$. For the third term, we know that $n-4m+2\leq -m-2$, and $m\geq5$, so $n-4m+2\leq-7$. We have $\varphi^{-7} < 0.035$. Combining the first three terms together, we get
\[0.09(0.002 + 2.24 + 0.035) < 0.205.\]

Directly following from the first condition in the lemma's statement, the term $\varphi^{n-3m+2} \le \varphi^{-2} < 0.382$. For the last term we use the condition $n - m \geq 0$ implying that $\varphi^{m-n-2} \le \varphi^{-2} < 0.382$. Combining all estimates together, we get
\[\left|\frac{F_{n+2}-1}{F_m}-L_{n-m+2}\right| \leq 0.205 + 0.382 + 0.382 = 0.969 < 1.\]
\end{proof}

\begin{corollary}
\label{cor:psdiv}
If $n-3m\leq-4$ and $n\geq11$ and $n - m \geq 0$, and $F_m|F_{n+2} -1$, then $\frac{F_{n+2}-1}{F_m} = L_{n-m+2}$.
\end{corollary}

\begin{proof}
We know that $\frac{F_{n+2}-1}{F_m}$ and $L_{n-m+2}$ are integers. By Lemma~\ref{lemma:estimate}, they are less than 1 apart. That means they are equal to each other.
\end{proof}

\subsection{Maximality}
\label{sec:max}

Now, we are ready for the theorem.

\begin{theorem}
\label{thm:Fibmax}
With the exception of $n=3$, we have that $m_{4k+1} = m_{4k+3} = m_{4k+4} = 2k+2$ and $m_{4k+2} = 2k+3$. In addition, we have $z_{4k+1} = z_{4k+2} = L_{2k+1}$, $z_{4k+3} = L_{2k+3}$, and $z_{4k+4} = L_{2k+4}$.
\end{theorem}

\begin{proof}
The divisibility pattern corresponds to Eq.~\ref{eq:2k+2|2k+4}--\ref{eq:2k+2|2k+3}.

If $n=3$, then $k=0$, and the general pattern gives us $S^F_3 = 4 =F_2 L_3 = 1 \cdot 4$. Unlike all other cases, the next Fibonacci number $F_3$ divides the Lucas number in the product $L_3$, so the partial sum $S^F_3$ is actually divisible by a greater Fibonacci number $F_3$: $S^F_3 = 2 \cdot 2 =2F_3$, which explains the exception for $n=3$. Coincidentally, 2 is a Lucas number, so we can write $S^F_3 = F_3 L_0$.

For $n \leq 10$, we verified by a program that our $m_n$ is maximal. Now, assume that $n \geq 11$.

Since $F_m|F_{n+2}-1$, we know that $(n+2)-m\geq 1$, and therefore $n-m\geq-1$. We leave it to the reader to show that for $n>2$, we have
\[ 1<\frac{F_{n+2} -1}{F_{n+1}} < 2.\]
Thus we can conclude that $\frac{F_{n+2}-1}{F_{n+1}}$ is not an integer, and $n-m\geq 0$.

Now, we consider the value of $n-3m$ for $n$ and $m$ from our statement. The values are summarized in Table~\ref{tab:n-3m}.

\begin{table}[ht!]
\begin{center}
\begin{tabular}{|c|c|c|}
\hline
$n$ 	& $m$		& $n-3m$ \\ \hline
$4k+1$	& $2k+2$	& $-2k-5$ \\	\hline
$4k+2$	& $2k+3$	& $-2k-7$ \\	\hline
$4k+3$	& $2k+2$	& $-2k-3$ \\	\hline
$4k+4$	& $2k+2$	& $-2k-2$ \\	\hline
\end{tabular}
\end{center}
\caption{Value of $n-3m$.}
\label{tab:n-3m}
\end{table}

In particular, for $n \geq 11$, we have $n-3m \leq -4$.

Now, for the sake of contradiction, suppose there exists an $m' > m$ such that $F_{m'} \mid F_{n+2} -1$. Then $m'$ must satisfy $n-3m' < n-3m \le -4$.

By Corollary~\ref{cor:psdiv}, for $n \ge 11$, we must have $\frac{F_{n+2} -1}{F_{m}}= L_{n-m+2}$ and $\frac{F_{n+2} -1}{F_{m'}}= L_{n-m'+2}$. Thus, we have an equation $F_mL_{n-m+2} = F_{m'}L_{n-m'+2}$. As $m' > m > 1$, and $n-m > n-m' \geq 4$, from Proposition~\ref{prop:FLIdentities}, we get $m = m'$, which is a contradiction.
\end{proof}

\subsection{Odd and triangular patterns}

When we started our experiments, we noticed we noticed a pattern for triangular numbers in the first column and a pattern for odd numbers in the second. Now, we understand that Table~\ref{tab:smalltrick} is extremely misleading because of its first two rows. It is true that the sum of one Fibonacci number divides the first Fibonacci number. However, because $F_1 = F_2$, the largest index is 2. We also now know that the value 3 in the second row is an exception to the pattern. Our triangular/odd pattern did not work out. Now, we can explain what happens when $n$ is triangular and, separately, when the index is odd.

Furthermore, in Table~\ref{tab:smalltrick}, we did not have a line for the sum of two elements. However, we know that the trick works for $n=2$. Why did we do this? In our original sequence, we set the condition that the index of the term being multiplied by a constant is less than the total number of terms. In other words, in the original trick, the term we multiply by has to be written out as part of the first $n$ terms. For $n=2$, we are using the third term, which was not allowed from the condition we set at the beginning because $3>2$.

\subsubsection{Triangular pattern}

We were interested in the cases where $n$ is a triangular number. To get the formulas for index $m$ and the corresponding Lucas number, we can take the $\frac{k(k+1)}{2}^\text{th}$ term from the second and third columns of Table~\ref{tab:trickAnyFib}. We can summarize the first few values in Table~\ref{tab:triangular}.

\begin{table}[ht!]
\begin{center}
\begin{tabular}{|c|c|c|c|}
\hline
Number of elements $n$ in the sum 	& index $m$ &	multiplier $z$	& index $i$ of $z$	\\ \hline
1	& 2	& 1	& 1 \\ \hline
3	& 3	& 2	& 0 \\ \hline
6	& 5	& 4	& 3 \\ \hline
10	& 7	& 11	& 5 \\ \hline
15	& 8	& 76	& 9 \\ \hline
21	& 12	& 199	& 11 \\ \hline
28	& 14	& 2207	& 16 \\ \hline
36	& 18	& 15127	& 20 \\ \hline
45	& 24	& 64079	& 23 \\ \hline
55	& 28	& 1149851	& 29 \\ \hline
\end{tabular}
\end{center}
\caption{Sum of the first $n$ Fibonacci numbers equals the $m$-th term times $z$, where $n$ is a triangular number.}
\label{tab:triangular}
\end{table}

The second column is a new sequence A372050
\[2,\ 3,\ 5,\ 7,\ 8,\ 12,\ 14,\ 18,\ 24,\ 28,\ 35,\ 41,\ 46,\ 54,\ 60,\ 68,\ 78,\ 89,\ 97,\ 107,\ 116,\ \ldots .\]


The fourth column is indices of Lucas numbers, which is a new sequence A372051:
\[1,\ 0,\ 3,\ 5,\ 9,\ 11,\ 16,\ 20,\ 23,\ 29,\ 33,\ 39,\ 47,\ 53,\ 62,\ 70,\ 77,\ 87,\ 95,\ 105,\ 117,\ \ldots .\]

\subsubsection{Odd pattern}

Our pattern with the first column of Table~\ref{tab:smalltrick} being triangular numbers and the third column being odd numbers did not continue. However, in Table~\ref{tab:trickAnyFib}, we have a pattern of period 4 for the multiplier $m$, where in each block of 4, there is exactly one odd number in the column for $m$. It corresponds to $n$ with remainders 2 modulo 4. Thus, we get a pattern $S^F_{4k+2} = F_{2k+3}L_{2k+1}$.

Initially, we did not include $n=2$, which corresponds to $m$ being odd, for reasons already stated. Instead, we had a special case of $n = 3$. However, when we include a row for $n = 2$ and exclude rows $n = 1$ and $n = 3$, the new Table~\ref{tab:odd} shows a really nice pattern, with the second column being odd numbers.

\begin{table}[ht!]
\begin{center}
\begin{tabular}{|c|c|c|c|}
\hline
 Number of elements $n$ in the sum 	& index $m$ & multiplier $z$ & index $i$ of $z$ \\ \hline
2	& 3	& 1 	& 1 \\ \hline
6	& 5	& 4 	& 3 \\ \hline
10	& 7	& 11	& 5 \\ \hline
14	& 9	& 29	& 7 \\ \hline
18	& 11	& 76	& 9 \\ \hline
22	& 13	& 199	& 11 \\ \hline
26	& 15	& 521	& 13 \\ \hline
30	& 17	& 1364	& 15 \\ \hline
34	& 19	& 3571	& 17 \\ \hline
\end{tabular}
\end{center}
\caption{Sum of the first $n$ Fibonacci numbers equals the $m$th term times $z$, where the second column is odd.}
\label{tab:odd}
\end{table}

\subsubsection{Combining odd and triangular patterns}

Furthermore, we can find a new pattern where both $n$ is a triangular number greater than 3, and the index $m$ is odd. As we show later in Corollary~\ref{cor:Fib-like}, when $m$ is odd, the corresponding trick can be performed for any Fibonacci-like sequence. These entries correspond to all positive integers $n \equiv 2 \pmod{4}$ that are of the form $\frac{k(k+1)}{2}$. The first few rows are shown in Table~\ref{tab:to}.

\begin{table}[ht!]
\begin{center}
\begin{tabular}{|c|c|c|c|}
\hline
 $n$ & $m$ & mutliplier $z$ & $i$ \\ \hline
6 & 5 & 4 & 3 \\ \hline
10 & 7 & 11 & 5 \\ \hline
66 & 35 & 7881196 & 33 \\ \hline
78 & 41 & 141422324 & 39 \\ \hline
190 & 97 & 71420983074726546239 & 95 \\ \hline
210 & 107 & 8784200221406821330636 & 105 \\ \hline
378 & 191 & 3152564691982405848945267213740827495676 & 189 \\ \hline
\end{tabular}
\caption{Triangular $n$, odd $m$.}
\label{tab:to}
\end{center}
\end{table}

We have previously shown that if $n = 4k + 2$, then $m = 2k + 3$ and $z = L_{2k + 1}$. It follows that the last column of values of $i$ is half of the first column, and the second column for $m$ is the last column plus 2.


Here are the first few numbers of the last column. This is now sequence A372718:
\[\ 3,\ 5,\ 33,\ 39,\ 95,\ 105,\ 189,\ 203,\ 315,\ 333,\ 473,\ 495,\ 663,\ 689,\ \ldots.\]

The triangular numbers modulo 4 form the sequence with period 8: 0, 1, 3, 2, 2, 3, 1, 0, 0 at which point it repeats. Thus, the indices $k$ corresponding to values 2 modulo 4 must be either 3 or 4 $\pmod{8}$. That means Table~\ref{tab:to} starts with two consecutive triangular numbers 6 and 10, then skips the next 6 triangular numbers, then has two consecutive triangular numbers, and so on.

\section{Fibonacci-like sequences}
\label{sec:flike}

\subsection{General case}

We are interested in the following question. When can a trick that works for the Fibonacci sequence be extended to any Fibonacci-like sequence?

\begin{theorem}
\label{thm:Fib-like}
The trick for the Fibonacci numbers using maximal $m_n$ can be extended to any Fibonacci-like sequence if $m_{n-1} +1 = m_{n}$ and $z_n = z_{n-1}$.
\end{theorem}

\begin{proof}
Consider the shifted Fibonacci sequence $F_{n-1}$. The sum of the first $n$ terms of this sequence is $F_{n+1} - 1$, which by our assumption equals $zF_{m_{n-1}} = zF_{m_{n}-1}$. But $F_{m_{n}-1}$ is the $m_n$th term of the shifted Fibonacci sequence. It follows that both sums of the Fibonacci sequence and shifted Fibonacci sequence are divisible by their corresponding $m$th term with the same ratio $z$. Given that these two sequences are linearly independent, this property can be extended by linearity to any Fibonacci-like sequence.
\end{proof}

\begin{corollary}
\label{cor:Fib-like}
The trick for the Fibonacci numbers using maximal $m$ can be extended to any Fibonacci-like sequence if and only if $n$ equals 3, or $n \equiv 2 \pmod{4}$.
\end{corollary}

\begin{proof}
We need to have $m_n-1 = m_{n-1}$, which only happens when $n \equiv 2 \pmod{4}$ due to Theorem~\ref{thm:Fibmax}.

We need to consider the case of $n=3$ separately, as it was a separate case for the Fibonacci numbers. In this case, we can manually check that the sum of the first three terms is always twice the third term. So it works.
\end{proof}

\subsection{Lucas numbers}
\label{sec:lucas}

Lucas numbers $L_n$, which is sequence A000032 in the OEIS \cite{OEIS}, follow the same recurrence as the Fibonacci numbers: $L_n = L_{n-1} + L_{n-2}$. The initial terms are $L_0 = 2$ and $L_1 = 1$:
\[2,\ 1,\ 3,\ 4,\ 7,\ 11,\ 18,\ 29,\ 47,\ 76,\ 123,\ 199,\ 322,\ 521,\ 843,\ \ldots .\]

We denote the sum of the first $n$ Lucas numbers, excluding the 0th term, as $S^L_n = \sum_{i=1}^nL_i$. This sum equals $L_{n+2} - 3$ and is sequence A027961 in the OEIS \cite{OEIS}. It starts as
\[1,\ 4,\ 8,\ 15,\ 26,\ 44,\ 73,\ 120,\ 196,\ 319,\ 518,\ 840,\ 1361,\ 2204,\ 3568,\ 5775,\ \ldots .\] 

What is the largest $m$ such that $S^L_n = \sum_{i=1}^nL_i$ is divisible by $L_m$? We wrote a program to calculate the answer for small values of $n$, and the results are summarized in Table~\ref{tab:lucas1}.

\begin{table}[ht!]
\begin{center}
\begin{tabular}{|c|c|c|}
\hline
Number of elements $n$ in the sum & index $m$ & multiplier $z$ \\
\hline
1 & 1 & 1 \\
\hline
2 & 3 & 1 \\
\hline
3 & 3 & 2 \\
\hline
4 & 2 & 5 \\
\hline
5 & 1 & 26 \\
\hline
6 & 5 & 4 \\
\hline
7 & 1 & 73 \\
\hline
8 & 3 & 30 \\
\hline
9 & 4 & 28 \\
\hline
10 & 7 & 11 \\ 
\hline
11 & 4 & 74 \\
\hline
12 & 4 & 120 \\
\hline
13 & 1 & 1361 \\
\hline
14 & 9 & 29 \\
\hline
15 & 3 & 892 \\
\hline
16 & 5 & 525 \\
\hline
17 & 1 & 9346 \\
\hline
18 & 11 & 76 \\
\hline
19 & 1 & 24473 \\
\hline
20 & 6 & 2200 \\
\hline
\end{tabular}
\end{center}
\caption{Divisibility of $S^L_n$.}
\label{tab:lucas1}
\end{table}

For $n\equiv2\pmod{4}$, we have the same pattern as for any Fibonacci-like sequence, as expected: $S^L_{4k+2} = L_{2k+3}L_{2k+1}$.

For $n = 4k$, the sum is $L_{4k+2} - 3$. Then, from a well-known equation $L_{2n}+3(-1)^{n}=5F_{n+1}F_{n-1}$, substituting $2k+1$ for $n$, we get $L_{4k+2}-3=5F_{2k+2}F_{2k} = 5F_{k+1}F_{2k}L_{k+1}$. Thus, it is always divisible by $L_{k+1}$. We also see that it is always divisible by 5. After dividing the multiplier by 5 and $L_{k+1}$, we get the following sequence, which is now sequence A372225:
\[1,\ 6,\ 24,\ 105,\ 440,\ 1872,\ 7917,\ 33558,\ 142120,\ 602085,\ 2550384,\ \ldots. \]

We wrote a program to calculate the values for $m$ when $n$ is odd. In the result, all the values were either $1$, $3$, or $4$. Moreover, the values seem to cycle with cycle length 24:
\[1,\ 3,\ 1,\ 1,\ 4,\ 4,\ 1,\ 3,\ 1,\ 1,\ 3,\ 1,\ 4,\ 4,\ 1,\ 1,\ 3,\ 1,\ 1,\ 3,\ 4,\ 4,\ 3,\ 1.\]
Let this sequence be $x_i$, where the index $i$ starts from 1. We wish to show that for any $n = 2k-1$, the sum $L_{n+2} - 3 = L_{2k+1} - 3 $ of the first $n$ Lucas numbers is divisible by $L_{x_k}$.

\begin{proposition}
The value $\frac{L_{2k+1} - 3}{L_{x_k}}$ is an integer.
\end{proposition}
\begin{proof}
For $k$ equaling 0, 1, 3, 4, 7, 9, 10, 12, 15, 16, 18, and 19, modulo 24, divisibility is trivial as $x_{k} = 1$, and $L_{1} = 1$.

For $k$ equaling 2, 8, 11, 17, 20, and 23 modulo 24, we have $x_{k} = 3$, and $L_{x_{k}} = 4$. The Lucas sequence modulo 4 has a period of 6 and is 2, 1, 3, 0, 3, 3. It follows that the sum of the first odd number of $n = 2k - 1$ terms is divisible by 4 if and only if $n\equiv3$ modulo 6, or equivalently, $k \equiv 2$ modulo 3. All our values of $k$ are as such.

For $k$ equaling 5, 6, 13, 14, 21, and 22 modulo 24, we have $x_{k} = 4$ and $L_{x_{k}} = 7$. The Lucas sequence modulo 7 has a period of 16 and is 2, 1, 3, 4, 0, 4, 4, 1, 5, 6, 4, 3, 0, 3, 3, 6. It follows that the sum of the first odd number $n = 2k - 1$ of terms is divisible by 7 if and only if $n\equiv9$ or 11 modulo 16, or equivalently, $k \equiv 5$ or 6 modulo 8. All our values of $k$ are as such.

Thus, this shows divisibility for all values of $k$.
\end{proof}

We will not be looking into maximality for odd $n$. On the other hand, before proving maximality for even $n$, we need to look at the products of Lucas numbers. Similar to Proposition~\ref{prop:FLIdentities} regarding the products of a Fibonacci and Lucas numbers, the products of two Lucas numbers decompose into the products of two Lucas numbers almost uniquely.

\begin{proposition}
\label{prop:LLIdentities}
If $L_a L_b = L_c L_d$, and the indices are nonnegative, then one of the following cases is true.
\begin{enumerate}
\item $a = c$ and $b = d$, corresponding to the equality $L_a L_b = L_a L_b$.
\item $a = d$ and $b = c$, corresponding to the equality $L_a L_b = L_b L_a$.
\item $(a,b,c,d)=(0,0,1,3)$, $(0,0,3,1)$, $(1,3,0,0)$, or $(3,1,0,0)$, corresponding to the equality $L_0 L_0 = L_1 L_3$, which is $2 \cdot 2 = 1 \cdot 4$.
\end{enumerate}
\end{proposition}

\begin{proof}
The product of two Lucas numbers $L_a$ and $L_b$ equals $L_{a+b}+L_{a-b}(-1)^b$. We thus have
\[L_{a+b}+(-1)^bL_{a-b}=L_{c+d}+(-1)^dL_{c-d},\]
which we can rearrange to obtain
\begin{equation}
\label{eq:La+b}
L_{a+b}= L_{c+d}+ (-1)^dL_{c-d}-(-1)^bL_{a-b}.
\end{equation}

Without loss of generality, we assume that $a\geq b$, $c\geq d$, and $a+b\geq c+d.$ 

Suppose $b \geq 1$, $c+d \leq a +b-2$, and $d \geq 2$. We look at the right-hand side of Eq.~\ref{eq:La+b}. We have $L_{c+d}\leq L_{a+b-2}$ because by the conditions, $c\geq d\geq2$, so $c+d\geq4$ and the Lucas sequence is strictly increasing starting from term 1. We also have $L_{a-b}\leq L_{a+b-2}$ because the difference between the indices is either 0, resulting in equality, or an even difference of at least 2, resulting in the correct inequality. For any two Lucas numbers of nonnegative indices with a difference of at least 2, the Lucas number with the greater index is strictly greater than the other Lucas number. So we have
\[L_{c-d}=L_{(c+d)-2d}< L_{(a+b-2)-2d+2}=L_{a+b-2d} \leq L_{a+b-4}.\]
We can conclude that $L_{c-d}<L_{a+b-4}$. Thus, we have:
\[L_{a+b}=L_{c+d}+(-1)^dL_{c-d}-(-1)^bL_{a-b}\leq L_{c+d}+L_{c-d}+L_{a-b}\]
\[< L_{a+b-2}+L_{a+b-4}+L_{a+b-2}<L_{a+b-1}+L_{a+b-2}=L_{a+b},\]
which has no solutions.

We are left to check separate cases $b <1$, $a+b = c+d$, $a+b = c+d+1$, and $d<2$.

\textbf{Case 1:} $a+b=c+d$. We have that $(-1)^bL_{a-b}=(-1)^dL_{c-d}.$ As Lucas numbers with nonnegative indices are unique, we have that $a-b=c-d$. Because we also have that $a+b=c+d$, it follows that $b=d$, and thus $a=c$ as well, corresponding to Cases 1 and 2.

From now on, we assume that $a+b> c+d$.

\textbf{Case 2:} $b=0$. We have $2L_{a}= L_{c+d}+ (-1)^dL_{c-d}$. As both $c+d$ and $c-d$ are less than $a+b$, the case is impossible.

From now on, we assume that $a+b> c+d$ and $b>0$.

\textbf{Case 3:} $c,d\leq1$. We must have $L_aL_b=$ 1, 2, or 4, where the nontrivial solutions are $(a,b)=(1,3)$ or $(3,1)$ while $(c,d) = (0,0)$. This corresponds to Case 3.

\textbf{Case 4:} $c\geq2$ and $d=1$. We have that $L_aL_b=L_c$ and $a+b > c+1$. Thus, we have $L_{a+b}+(-1)^bL_{a-b}=L_c$. As $b > 0$, by parity considerations we have $a+b > a-b + 1$. By the growth rate of Lucas numbers, $L_n$ cannot be a sum of two Lucas numbers with indices less than $n-1$, with the exception of $n=2$, which we covered before.

\textbf{Case 5:} $c\geq2$ and $d=0$. We have that $L_aL_b=2L_c$ and $a+b > c$. Thus, we have $L_{a+b}+(-1)^bL_{a-b}=2L_c$. We have $a-b < a+b -1$. We now consider the case when $c < a+b -2$. Then $L_{a-b} + 2L_c \leq L_{a+b-2} + 2L_{a+b-3} = L_{a+b-1} + L_{a+b-3} < L_{a+b}$. Thus, this case is impossible. We are left to check the case when $b >0$ and $c = a+b-2$, or $c= a+b-1$. We have 
\[L_{a+b}+(-1)^bL_{a-b}=2L_{a+b-1} \quad \textrm{ or } \quad L_{a+b}+(-1)^bL_{a-b}=2L_{a+b-2}.\]
After minor manipulation, we get correspondingly
\[(-1)^bL_{a-b}=L_{a+b-3} \quad \textrm{ or } \quad (-1)^bL_{a-b}=-L_{a+b-3},\]
which are impossible due to parity considerations.

What is left to check is the case $a+b = c+d +1$, $b>0$ and $d >1$.

\textbf{Case 6:} $a+b = c+d +1$, $b>0$ and $d >1$. We can rewrite Eq.~\ref{eq:La+b} as 
\[L_{a+b-2}= (-1)^dL_{c-d}-(-1)^bL_{a-b}.\]
If $b=1$, we get $(-1)^dL_{c-d} = 0$, which is impossible. Suppose $b>1$. Then $a-b \leq a+b -4$. Also, $c-d \leq a+b-5$. Thus, $(-1)^dL_{c-d}-(-1)^bL_{a-b} \leq L_{a+b-4} + L_{a+b-5} = L_{a+b-3}$, which is impossible.
\end{proof}

We have the following lemma, which is an analog of Lemma~\ref{lemma:estimate}.

\begin{lemma}
\label{lem:Lucasmcond}
If $n-3m\leq-4$ and $n\geq11$ and $n-m\geq1$, then $\left|\frac{L_{n+2}-3}{L_m}-L_{n-m+2}\right|<1$.
\end{lemma}

\begin{proof}
Using Binet's formulas, we get that
\[ \frac{L_{n+2}-3}{L_m}-L_{n-m+2}=\frac{\varphi^{n+2}+(\frac{-1}{\varphi})^{n+2}-3}{\varphi^{m}+(\frac{-1}{\varphi})^{m}}-\left(\varphi^{n-m+2}+\left(\frac{-1}{\varphi}\right)^{n-m+2}\right)\]
\[=\frac{\varphi^{n+2}+(\frac{-1}{\varphi})^{n+2}-3}{\varphi^{m}+(\frac{-1}{\varphi})^{m}}-\varphi^{n-m+2}-\left(\frac{-1}{\varphi}\right)^{n-m+2}\]
\[=\frac{\varphi^{n+2}+(\frac{-1}{\varphi})^{n+2}-3-\varphi^{n+2}-(\frac{-1}{\varphi})^{m}\varphi^{n-m+2}}{\varphi^{m}+(\frac{-1}{\varphi})^{m}}-\left(\frac{-1}{\varphi}\right)^{n-m+2}\]
\[=\frac{(\frac{-1}{\varphi})^{n+2}-3-(-1)^{m}\varphi^{n-2m+2}}{\varphi^{m}+(\frac{-1}{\varphi})^{m}}-\left(\frac{-1}{\varphi}\right)^{n-m+2}\]
\[=\frac{(\frac{-1}{\varphi})^{n+2}-3+\varphi^{n-4m+2}}{\varphi^{m}+(\frac{-1}{\varphi})^{m}}-(-1)^m\varphi^{n-3m+2}-\left(\frac{-1}{\varphi}\right)^{n-m+2}\]
\[=\frac{(\frac{-1}{\varphi})^{n+2}}{\varphi^{m}+(\frac{-1}{\varphi})^{m}}-\frac{3}{\varphi^{m}+(\frac{-1}{\varphi})^{m}}+\frac{\varphi^{n-4m+2}}{\varphi^{m}+(\frac{-1}{\varphi})^{m}}-(-1)^m\varphi^{n-3m+2}-\left(\frac{-1}{\varphi}\right)^{n-m+2}.\]

Therefore, 
\[\left|\frac{L_{n+2}-3}{L_m}-L_{n-m+2}\right|\leq\frac{{\varphi}^{-n-2}}{\varphi^{m}+(\frac{-1}{\varphi})^{m}}+\frac{3}{\varphi^{m}+(\frac{-1}{\varphi})^{m}}+\frac{\varphi^{n-4m+2}}{\varphi^{m}+(\frac{-1}{\varphi})^{m}}+\varphi^{n-3m+2}+\varphi^{m-n-2}.\]

From the two conditions $n-3m\leq-4$ and $n\geq11$, we get $\frac{n+4}{3}\leq m$, and $m\geq5$. We estimate the terms above one by one. For $m\geq 5$, we have $\frac{1}{\varphi^{m}+(\frac{-1}{\varphi})^{m}}<0.091$. For $n\geq11$, we have $\varphi^{-n-2}<0.002$. We know that $n-4m+2\leq-m-2$ and $m\geq5$, so $n-4m+2\leq-7$. We have $\varphi^{-7}<0.035$. Combining the first three terms together, they do not exceed
\[0.091(0.002+3+0.035)<0.277.\]
From $n-3m\leq-4$, we get that $n-3m+2\leq-2$, so the term $\varphi^{n-3m+2}\leq\varphi^{-2}<0.382$. From the condition $n-m\geq1$, it is implied that the term $\varphi^{m-n-2}\leq\varphi^{-3}<0.237$.

Combining all estimates together, we get
\[\left|\frac{L_{n+2}-3}{L_m}-L_{n-m+2}\right|\leq0.277+0.382+0.237=0.896<1.\]
\end{proof}

This leads us to the analog of Theorem~\ref{thm:Fibmax}, where we denote the maximum index of a Lucas number that divides the sum of the first $n$ Lucas numbers as $m^L_n$.

\begin{theorem}
We have $m^L_{4k+2} = 2k+3$.
\end{theorem}

\begin{proof}
For $n \leq 10$, we verified by a program that our $m$ is maximal. Now assume $n \geq 11$.

From Corollary~\ref{cor:Fib-like}, we have $S_{4k+2}^L = L_{2k+3}L_{2k+1}$, which means $m^L_{4k+2} \geq 2k+3$. Next, for the sake of contradiction, suppose there exists an $m' > 2k+3$ such that $L_{m'}|L_{n+2}-3$. We start by showing that $m'$ must be smaller than $n$. First, $m'$ cannot be greater than $n+1$, as, otherwise, $L_{m'}>L_{n+2}-3 = S_{n}^L$. If $m'=n+1$, then $\frac{L_{n+2}-3}{L_{n+1}} =1+\frac{L_{n}-3}{L_{n+1}}$, and for $n>2$, we have $0<\frac{L_{n}-3}{L_{n+1}}<1$, so $\frac{L_{n+2}-3}{L_{n+1}}$ cannot be an integer. If $n=m'$, then $\frac{L_{n+2}-3}{L_{n}}=1+\frac{L_{n+1}-3}{L_{n}}=2+\frac{L_{n-1}-3}{L_{n}}$, and for $n>3$, we have $0<\frac{L_{n-1}-3}{L_{n}}<1$, so $\frac{L_{n+2}-3}{L_{n}}$ cannot be an integer. So given that $n>10$, we can conclude that $n-m' \geq1$.

Now, we consider the value of $n-3m'$. By our assumption $m' \geq 2k+3$. We conclude that $n-3m'\leq-4$. It follows that $m'$ satisfies the conditions of Lemma~\ref{lem:Lucasmcond}, implying that $L_{n+2}-3 = L_{m'}L_{n-m'+2}$. Finally, from Proposition~\ref{prop:LLIdentities}, the decomposition of $L_{n+2}-3$ into the product of two Lucas numbers must be unique, and thus $m'=m=2k+3$.
\end{proof}

Before proving the maximality for $n=4k$, we need the following lemma. For any two real numbers $a$ and $b$, let us say that $a\equiv b$ if and only if the fractional parts of $a$ and $b$ are equivalent.
\begin{lemma}
For any two integers $a$ and $b$ and a constant $c$, the following is true: 
\[\frac{c+L_a}{L_b}\equiv \frac{c-(-1)^bL_{2b-a}}{L_b}.\]
\end{lemma}

\begin{proof}
We have
\[\frac{c+L_a}{L_b}\equiv \frac{c+L_a}{L_b}-L_{a-b}=\frac{c+L_{a}-L_bL_{a-b}}{L_b}.\]
By the well-known identity $L_nL_m=L_{n+m}+(-1)^nL_{n-m}$, we get
\[\frac{c+L_{a}-L_bL_{a-b}}{L_b} = \frac{c+L_a-(L_a+(-1)^bL_{2b-a})}{L_b}=\frac{c-(-1)^bL_{2b-a}}{L_b}.\]
\end{proof}

\begin{theorem}
We have $m^L_{4k} = k+1$.
\end{theorem}

Because $n-3m = 4k-3(k+1)$ is not less than or equal to $-4$ for $k>1$, we cannot guarantee that $\frac{L_{4k+2}-3}{L_{k+1}}$ is a Lucas number using Lemma~\ref{lem:Lucasmcond}, thus we cannot use the lemma about unique Lucas number products either. Therefore, we must use an alternative method to prove this.

\begin{proof}
First of all, $S^{L}_{4k}$ is divisible by $L_{k+1}$ because
\[ S^{L}_{4k} = L_{4k+2}-3 = 5F_{2k+2}F_{2k} = 5F_{k+1}F_{2k}L_{k+1}.\]
Now we prove the maximality. Assume for the sake of contradiction that a positive integer $x$ exists such that $L_{k+1+x}|L_{4k+2}-3$. Using the lemma above, we get

\[\frac{-3+L_{4k+2}}{L_{k+1+x}} \equiv \frac{-3-(-1)^{k+1+x}L_{2x-2k}}{L_{k+1+x}}\equiv0.\]

We use the property that $L_{-n}=(-1)^nL_n$, then we multiply by $(-1)^{k+x}$ to get

\[\frac{-3-(-1)^{k+1+x}L_{2x-2k}}{L_{k+1+x}}\equiv\frac{-3-(-1)^{k+1+x}L_{2k-2x}}{L_{k+1+x}}\equiv\frac{3(-1)^{k+1+x}+L_{2k-2x}}{L_{k+1+x}}\equiv 0.\]

We apply the above lemma again, then we multiply by $(-1)^{k+x}$ to get

\[\frac{3(-1)^{k+x+1}+L_{2k-2x}}{L_{k+1+x}} \equiv \frac{3(-1)^{k+1+x}-(-1)^{k+1+x}L_{4x+2}}{L_{k+1+x}}\equiv \frac{-3+L_{4x+2}}{L_{k+1+x}}\equiv0\]

Thus, we have 
\begin{equation} \label{eq:lucaspf}
\frac{L_{4x+2}-3}{L_{k+1+x}} \equiv 0.
\end{equation}
Note that Eq.~\ref{eq:lucaspf} is different from our divisibility assumption $L_{k+1+x}|L_{4k+2}-3$ as we have $L_{4x+2}$ (not $L_{4k+2}$) in the numerator. Now we check whether we can apply Lemma~\ref{lem:Lucasmcond} on Eq.~\ref{eq:lucaspf}.

The first condition $n-3m\leq-4$ corresponds to $4x+2 - 3(k+x+1) \le -4$, which is equivalent to $x \le 3k - 1$. The original assumption $L_{k+1+x}|L_{4k+2}-3$ implies $k+1+x \leq 4k+1$, which is equivalent to $x \le 3k$. Moreover, $k+1+x$ cannot be $4k+1$ as that would imply $L_{4k+1}|L_{4k+1}+L_{4k}-3$, i.e., $L_{4k+1}|L_{4k}-3$, which is impossible for large $k$. Thus, the first condition is satisfied.

The second condition $n\geq11$ is satisfied when $x \ge 3$. If $x < 3$, we have two cases of $x$ equaling 1 or 2. Thus, the numerator in Eq.~\ref{eq:lucaspf} is one of $L_6 - 3 = 15$ and $L_{10}-3=120$. The largest Lucas number that is a factor is $L_3 = 4$, corresponding to $k < 3$, the case covered by our program.

The last condition $n-m\geq 1$ corresponds to $4x+2 - (k+x+1) \geq 1$, which is equivalent to  $k+1+x \leq 4x+1$, which is satisfied due to Eq.~\ref{eq:lucaspf}.

So all conditions of Lemma~\ref{lem:Lucasmcond} are satisfied, and we get 
\[ L_{4x+2}-3 = L_{k+1+x}L_{3x-k+1}. \]
However, on the other hand, using a well-known equation, we have that
\[ L_{4x+2}-3 = L_{(2x+2)+2x} - (-1)^{2x}L_{(2x+2)-2x} = 5F_{2x+2}F_{2x} \]
has a factor 5, while no Lucas number can have a factor 5 (as a Lucas number mod 5 can only be 2, 1, 3, and 4). So the equation above is impossible and we have a contradiction.
\end{proof}

\section{General Lucas sequences}
\label{sec:gen}

\subsection{Definitions}

The Lucas sequences $U_{n}(P,Q)$ and $V_{n}(P,Q)$ are integer sequences that satisfy the recurrence relation
$x_{n}=P x_{n-1} - Q x_{n-2},$
where $P$ and $Q$ are fixed integers. We also have initial conditions: $U_{0}(P,Q) = 0$, $U_{1}(P,Q) = 1$, $V_{0}(P,Q) = 2$, and $V_{1}(P,Q) = P$. Sometimes sequences $U_n$ are called Lucas sequences of the first kind, while sequences $V_n$ are called Lucas sequences of the second kind. Sometimes, we drop parameters $P$ and $Q$ and write $U_n$ for our sequence. When $(P,Q) = (1,-1)$ we get Fibonacci-like sequences. It is natural to ask whether our trick for Fibonacci-like sequences can be extended to other Lucas sequences.

To start, we need to calculate the sum of the first $n$ terms of $U_n(P, Q)$ beginning from index 1. We denote this sum as $S_n^U$, and it equals the following
\begin{equation}
S_n^U = \sum_{i=1}^{n} U_i(P, Q)=\frac{QU_{n}-U_{n+1}+1}{Q-P+1}.
\label{eq:USum}
\end{equation}

Consider the shifted sequence, where the $n$th term is $U_{n-1}(P, Q)$. Its sum of the first $n$ terms is $S_{n-1}^U$.

\begin{theorem}
\label{thm:generalrule}
If $S_n^U = U_{m}z$ and $S_{n-1}^U = U_{m-1}z$, then the trick works for any Lucas sequence with the same recurrence as $U$.
\end{theorem}

\begin{proof}
Any Lucas sequence with the same recursion is a linear combination of $U_n$ and $U_{n-1}$. Thus, the trick works.
\end{proof}

\subsection{General discussion}

We are looking at the divisibility of expression
\[\frac{QU_{n}-U_{n+1}+1}{Q-P+1}.\]
Thus, it is useful to check when it simplifies. 

One way in which it simplifies us when the denominator's absolute value is 1: $Q-P+1 = \pm 1$. This happens when $Q=P$, or $Q = P-2$, and we get
\[S_n^U = \pm (QU_{n} - U_{n+1}+1).\]

The other case in which is simplifies is when $P=1$; in this case we get
\[S_n^U = \frac{-U_{n+2}+1}{Q}.\]

Both conditions $P = 1$ and $Q-P+1 = \pm 1$ are met for $U_n(1,-1)$ and $U_n(1,1)$, where the first sequence is Fibonaccis. We look at $U_n(1,1)$ in Section~\ref{sec:U11}.

We also consider examples where only one condition is met: a) a sequence $U(-1,-1)$ for the case $Q=P$ in Section~\ref{sec:U-1-1}; b) a sequence $U(3,1)$ for the case $Q=P-2$ in Section~\ref{sec:U31}.

We also look at some famous general Lucas sequences. For Jacobsthal numbers $J_n = U_n(1,-2)$ we have $P=1$ and $Q-P+1 = -2$. But it is easy to be divisible by 2, so we studied this sequence in Section~\ref{sec:jacob}. Another famous sequence is Pell numbers $P_n = U_n(2,-1)$, which we study in Section~\ref{sec:pell}.

We already covered Fibonacci and Lucas numbers. We continue with the other sequences below.

\subsection{Examples}

\subsubsection{$U(-1, -1)$}
\label{sec:U-1-1}

Sequence $U(-1, -1)$ is sequence A039834, which is the sequence of signed Fibonacci numbers, $U_n(-1,-1) = (-1)^{n-1}F_n$:
\[0,\ 1,\ -1,\ 2,\ -3,\ 5,\ -8,\ 13,\ -21,\ \ldots.\]

Partial sums $S_{n}$ starting from index 1 is sequence A355020
\[1,\ 0,\ 2,\ -1,\ 4,\ -4,\ 9,\ -12,\ 22,\ -33,\ \ldots,\]
which equals
\[S_{n} = (-1)^{n-1} F_{n-1} + 1 = -U_{n-1}+1.\]

As this sequence is similar to Fibonacci numbers, the proof of maximality follows from the proof of maximality for the Fibonacci sequence. We denote the maximum index for the term of the sequence $U$ that divides the sum of the first $n$ terms of this sequence as $m^U_n$.

\begin{theorem}
We have the maximum divisibility index $m^U_{4k} = m^U_{4k+2} = 2k$, $m^U_{4k+1} = 2k-1$, and $m^U_{4k+3} = 2k+2$. Correspondingly, the multipliers are $z^U_{4k} = L_{2k-1}$, $z^U_{4k+2} = L_{2k+1}$, $z^U_{4k+1} = L_{2k+1}$, and $z^U_{4k+3} = -L_{2k}$.
\end{theorem}

\begin{proof}
We have $S_{4k} = -(F_{4k-1} - 1) = -S^F_{4k-3} = -F_{2k}L_{2k-1} = U_{2k}L_{2k-1} $, where we used Eq.~\ref{eq:2k+2|2k+1}. Similarly, we have $S_{4k+2} = -(F_{4k+1} - 1)) = -S^F_{4k-1} = -F_{2k}L_{2k+1} = U_{2k}L_{2k+1} $, where we used Eq.~\ref{eq:2k+2|2k+3}. The maximality follows from the proof for the Fibonacci sequence.

Similarly, $S_{4k+1} = F_{4k} + 1 = F_{2k-1}L_{2k+1} = U_{2k-1}L_{2k+1}$ and $S_{4k+3} = F_{4k+2} + 1 = F_{2k+2}L_{2k} = -U_{2k+2}L_{2k}$. In both cases, we have $S_{2i+1} = F_{2i} + 1$. To prove maximality, notice that $S_{2i+1}$ differs from $S^F_{2i-2} = F_{2i} - 1$, by replacing 1 with $-1$. This change does not affect the argument in Lemma~\ref{lemma:estimate} because only the term $\frac{-\sqrt{5}}{\varphi^m-(\frac{-1}{\varphi})^m}$ changes sign, but we only consider the absolute value of each term, so the rest of the proof remains exactly the same.
\end{proof}

Now we check when $m$ increases by 1. We see that $m^U_{4k+2} = 2k = m^U_{4k+1} + 1$. Thus, the trick could be extended to the general sequence for $n = 4k+2$.

\subsubsection{$U(1,1)$}
\label{sec:U11}

Sequence $U(1,1)$ is sequence A128834, which is just a periodic sequence with a period of 6:
\[0,\ 1,\ 1,\ 0,\ -1,\ -1,\ 0,\ 1,\ 1,\ 0,\ -1,\ -1,\ \ldots.\]

The partial sums (starting from index 1) also form a periodic sequence with a period of 6
\[1,\ 2,\ 2,\ 1,\ 0,\ 0,\ 1,\ 2,\ 2,\ 1,\ 0,\ 0,\ \ldots,\]
which up to a shift is sequence A131026.

Partial sums are divisible by any non-zero element in the sequence. So, we can say that the largest index $m$ is infinity. In this case, we cannot use our method. We need to check the general case directly. Suppose a general sequence starts as $a$ and $b$ (from index 1), it continues as
\[a,\ b,\ b-a,\ -a,\ -b,\ -b+a,\ a,\ b,\ \ldots.\]
It always has period 6. The partial sums starting from index 1 are
\[a,\ a+b,\ 2b,\ -a+2b,\ -a+b,\ 0,\ a, \ldots.\]

Thus, there will always be a trick for $n \equiv 1,\ 3,\ 5,\ 6 \pmod{6}$ with the largest index $m$ at infinity.

\subsubsection{$U(3, 1)$}
\label{sec:U31}

Sequence $U(3, 1)$, which is sequence A001906 in the OEIS, is the bisection of the Fibonacci sequence. In other words A001906$(n) = F_{2n}$:
\[0,\ 1,\ 3,\ 8,\ 21,\ 55,\ 144,\ 377,\ \ldots.\]

The partial sums $S_n$ sequence is
\[1,\ 4,\ 12,\ 33,\ 88,\ 232,\, 599,\ \ldots,\]
which is A027941, defined as $S_n = F_{2n + 1} - 1$. It follows that $S_n = S^F_{2n -1}$, and we already studied the divisibility of this sequence by the Fibonacci numbers.

\begin{theorem}
For the sequence $U(3,1)$ and any $n>0$, the maximum $m$ such that $U_m$ divides the partial sum $S_n$ is $m=\lceil\frac{n}{2}\rceil$.
\end{theorem}

\begin{proof}
Consider $n = 2k$. We have $S_{2k} = S^{F}_{4k-1}$. By Theorem~\ref{thm:Fibmax}, the maximal Fibonacci number that divides it is $F_{2k} = U_{k}$.

Consider $n = 2k - 1$. We have $S_{2k-1} = S^{F}_{4k-3}$. By Theorem~\ref{thm:Fibmax}, the maximal value of $U$ for Fibonacci numbers is $F_{2k} = U_{k}$.
\end{proof}

Thus, when $n$ is odd $m_{n-1} + 1 = \frac{n-1}{2} + 1 = \frac{n+1}{2} = m_n$. Thus, for odd $n$, the trick works for any general sequence with the same recursion.

\subsection{Jacobsthal}
\label{sec:jacob}

Consider the Jacobsthal numbers, sequence A001045. The sequence is defined as $J_n = J_{n-1} + 2J_{n-2}$, with $J_0 = 0$, $J_1 = 1$:
\[0,\ 1,\ 1,\ 3,\ 5,\ 11,\ 21,\ 43,\ 85,\ 171,\ 341,\ 683,\ 1365,\ 2731,\ \ldots .\]

Partial sums of Jacobsthal numbers form sequence A000975. The sequence is defined recursively as $a(2n) = 2a(2n-1)$ and $a(2n+1) = 2a(2n)+1$, with initial terms $a(0) = 1$ and $a(1) = 1$. In an alternative definition, when a number in this sequence is written in binary, the binary digits alternate:
\[1,\ 2,\ 5,\ 10,\ 21,\ 42,\ 85,\ 170,\ 341,\ 682,\ 1365,\ 2730,\ 5461,\ \ldots .\]

The formula for the sum of the first $n$ Jacobshtal numbers $S^J_n$ is well-known, and also follows from Eq.~\ref{eq:USum}:
\[S^J_{2k-1} = \sum_{i=1}^{2k-1} J_i = J_{2k} \quad \textrm{ and } \quad S^J_{2k} = \sum_{i=1}^{2k} J_i = J_{2n+1} - 1 = 2J_{2k}.\]

From here we have the following theorem related to the sum of the first $n$ Jacobsthal numbers, where $m^J_n$ is the largest index such that the corresponding Jacobsthal number divides the given sum, and $z^J_n$ is the multiplier.

\begin{theorem}
For the Jacobshtal sequence we have the maximum divisibility index $m^J_{2k-1} = 2k$ and $m^J_{2k} = 2n$ while the multipliers are $z^J_{2k-1} =1$ and $z^J_{2k} = 2$.
\end{theorem}

\begin{proof}
We know that $S^J_{2k-1}= J_{2k}$. Since the Jacobsthal numbers form an increasing sequence, we know that $2k$ is the largest index less than or equal to $J_{2k}$, so we must have $m = 2k$ and $z = 1$.

Now, for the second part of the proof, we know that $S^J_{2k} = 2J_{2k} = J_{2k + 1} - 1$. Since $2J_{2k} < J_{2k + 1}$, so $J_{2k + 1}$ cannot divide the sum. Thus, we have $m = 2k$ and $z = 2$.
\end{proof}

We see that these sequences are one of the easiest to study in the framework of our trick. 

\begin{theorem}
A common trick for all Jacobsthal-like sequences does not exist.
\end{theorem}

\begin{proof}
Suppose a common trick exists for some $n$: $S^J_n = zJ_m$. Then it also exists for the shifted Jacobsthal sequence, and we have $S^J_{n-1} = zJ_{m-1}$. 

Suppose $n = 2k +1$. Then $S^J_{n} = S^J_{2k+1} = J_{2k+2}$ and $S^J_{n-1} = S^J_{2k} = 2J_{2k}$. Numbers $J_{2k}$ and $J_{2k+2}$ are coprime: if they share a divisor, then $J_{2k+1}$ and the rest of the sequence would share the same divisor. Therefore, $z = 1$. However, this implies that two consecutive sums equal to two consecutive terms correspondingly, which is not true. Thus, the trick cannot work for odd $n$.

Suppose $n=2k$ is even. Then we have $S^J_{n-1} = J_{2k} = zJ_{m-1}$ and $S^J_{n} = 2J_{2k} = zJ_{m}$. It follows that $J_m = 2J_{m-1}$, which is impossible as Jacobsthal numbers are odd.
\end{proof}

\subsection{Pell numbers}
\label{sec:pell}

We denote Pell numbers as $P_n$. Pell numbers are defined by recurrence relation $P_n = 2P_{n-1} + P_{n-2}$ for $n > 1$ and have initial conditions $P_0 = 0$, $P_1 = 1$. This is the Lucas sequence of the first kind $U_n(2,-1)$. In the database, this is sequence A000129. It starts as
\[0,\ 1,\ 2,\ 5,\ 12,\ 29,\ 70,\ 169,\ 408,\ 985,\ 2378,\ 5741,\ 13860,\ \ldots .\]

Partial sums of Pell numbers, denoted as $S^P_n$, are sequence A048739, which starts as
\[1,\ 3,\ 8,\ 20,\ 49,\ 119,\ 288,\ 696,\ 1681,\ 4059,\ 9800,\ 23660,\ \ldots .\]
We know that
\[S^P_n = \frac{P_{n}+P_{n+1}-1}{2}.\]

What is the largest $m^P_n$ such that $S^P_n$ is divisible by $P_m$? We summarize the results of our program in Table~\ref{tab:Pell}.

\begin{table}[ht!]
\begin{center}
\begin{tabular}{|c|c|c|}
\hline
Number of elements $n$ in the sum & index $m$ & multiplier $z$ \\ \hline
1	& 1	& 1 \\ \hline
2	& 1	& 3 \\ \hline
3	& 2	& 4 \\ \hline
4	& 3	& 4 \\ \hline
5	& 1	& 49 \\ \hline
6	& 1	& 119 \\ \hline
7	& 4	& 24 \\ \hline
8	& 5	& 24 \\ \hline
9	& 1	& 1681 \\ \hline
10	& 1	& 4059 \\ \hline
11	& 6	& 140 \\ \hline
12	& 7	& 140 \\ \hline
13	& 1	& 57121 \\ \hline
14	& 1	& 137903 \\ \hline
15	& 8	& 816 \\ \hline
16	& 9	& 816 \\ \hline
17	& 1	& 1940449 \\ \hline
18	& 1	& 4684659 \\ \hline
19	& 10	& 4756 \\ \hline
20	& 11	& 4756 \\ \hline
\end{tabular}
\end{center}
\caption{Pell numbers.}
\label{tab:Pell}
\end{table}

The table exhibits a pattern and implies the following maximality conjecture.

\begin{conjecture}
For the Pell sequence, we have $m^P_{4k+1} = m^P_{4k+2} = 1$, $m^P_{4k-1} = 2k$, and $m^P_{4k} = 2k+1$.
\end{conjecture}

To confirm the divisibility for $n$ with remainders 0 and 1 modulo 4, we use the following well-known equalities \cite{SD,Bradie}:
\[ S_{4k-1}^P=2P_{2k}^2 \quad \mathrm{ and }\quad S_{4k}^P=2P_{2k+1}P_{2k},\]
which implies that the Pell numbers have a nontrivial trick for such $n$. Now, we are ready for our theorem.

\begin{theorem}
For a Pell-like sequence $S_n$, the sum of the first $4k$ terms (starting from index 1) is the $2k+1$st term times $2P_{2k}$.
\end{theorem}

\begin{proof}
From our divisibility pattern, we know that for $n = 4k$, the sum of $n$ Pell numbers is divisible by $P_{2k+1}$ with the ratio $2P_{2k}$; we also know that the sum of the first $n-1$ terms is divisible by $P_{2k}$ with the same ratio. As the index decreases by 1, by our Theorem~\ref{thm:generalrule}, the trick works for any Pell-like sequence.
\end{proof}

\section{Tribonacci}
\label{sec:trib}

Tribonacci numbers $T_n$ are defined as $T_0 = T_1 = 0$, $T_2 = 1$, and $T_n = T_{n-1} + T_{n-2} + T_{n-3}$ for $n > 2$. This is sequence A000073, and it starts as
\[0,\ 0,\ 1,\ 1,\ 2,\ 4,\ 7,\ 13,\ 24,\ 44,\ 81,\ 149,\ 274,\ 504,\ 927,\ \ldots.\]

Similarly, we can define a Tribonacci-like sequence by positing that the initial terms are $a$, $b$, and $c$, and each consecutive term is the sum of the previous three terms. Table~\ref{tab:tri} shows the first ten terms and the first ten partial sums.

\begin{table}[ht!]
\begin{center}
\begin{tabular}{|c|c|c|}
\hline 
Index $n$ & $n$th term & Partial Sum of the first $n$ terms \\ \hline
1	& $a$		& $a$	\\ \hline
2	& $b$		& $a+b$	\\ \hline
3	& $c$		& $a+b+c$	\\ \hline
4	& $a+b+c$	& $2a+2b+2c$	\\ \hline
5	& $a+2b+2c$	& $3a+4b+4c$	\\ \hline
6	& $2a+3b+4c$	& $5a+7b+8c$	\\ \hline
7	& $4a+6b+7c$	& $9a+13b+15c$	\\ \hline
8	& $7a+11b+13c$	& $16a+24b+28c$	\\ \hline
9	& $13a+20b+24c$	& $29a+44b+52c$	\\ \hline
10	& $24a+37b+44c$	& $53a+81b+96c$	\\ \hline
\end{tabular}
\end{center}
\caption{Tribonacci-like sequence and its partial sums.}
\label{tab:tri}
\end{table}
We actually calculated more values than this. But we only found three cases when the sum divides one of the terms:

\begin{enumerate}
\item The sum of the first three terms is the fourth term.
\item The sum of the first four terms is twice the fourth term.
\item The sum of the first eight terms is 4 times the seventh term.
\end{enumerate}

Interestingly, we can generalize this for any $n$-nacci numbers. An $n$-nacci like sequence starts with any $n$ integers $a_1$, $a_2$, \ldots, $a_n$. After that, each consecutive term is the sum of the previous $n$ terms.

\begin{theorem}
For any $n$-nacci-like sequence
\begin{enumerate}
\item The sum of the first $n$ terms is the $(n+1)$st term.
\item The sum of the first $n+1$ terms is twice the $(n+1)$st term.
\item the sum of the first $2n+2$ terms is equal to 4 times the $(2n+1)$st term.
\end{enumerate}
\end{theorem}

\begin{proof}
We can calculate that the sum of the first $n$ terms is $a_1 + a_2 + \cdots + a_n$, which is the $(n + 1)$st term in the $n$-nacci-like sequence. The sum of the first $n + 1$ terms is $2a_1 + 2a_2 + \cdots + 2a_n$, which is $2$ times the $(n + 1)$st term in the $n$-nacci-like sequence. 

Consider the sum of the first $2n+2$ terms.
We have
\[\sum_{i=1}^{2n+2} a_{i} = \sum_{i=1}^{n} a_{i} + \sum_{i=n+1}^{2n} a_{i} + a_{2n+1} + a_{2n+2}.\]
We can replace $\sum_{i=1}^{n} a_{i}$ with $a_{n+1}$ and $\sum_{i=n+1}^{2n} a_{i}$ with $a_{2n+1}$, and also expand $a_{2n+2}$ to get
\[\sum_{i=1}^{2n+2} a_{i} = a_{n+1} + a_{2n+1} + a_{2n+1} + \sum_{i=n+2}^{2n+1} a_{i} = a_{n+1} + 3a_{2n+1} + \sum_{i=n+2}^{2n} a_{i}.\]
After redistributing, we have
\[\sum_{i=1}^{2n+2} a_{i} = 3a_{2n+1} + \sum_{i=n+1}^{2n} a_{i} = 4a_{2n+1}.\]
\end{proof}

\section{Acknowledgments}

We are grateful to the MIT PRIMES STEP program for giving us the opportunity to conduct this research.


\begin{thebibliography}{9}



\bibitem{Bradie} Brian Bradie, Extensions and refinements of some properties of sums involving Pell numbers, \textit{Missouri Journal of Mathematical Sciences} 22.1 (2010): 37--43.

\bibitem{FNT} Fibonacci Number Trick, available at \url{https://www.1728.org/fibonacci.htm}, accessed in March 2024.

\bibitem{TG} Letter from Toby Gee in Mathematical Spectrum, vol 29 (1996/1997), page 68.

\bibitem{OEIS} OEIS Foundation Inc. (2023), The On-Line Encyclopedia of Integer Sequences, Published electronically at \url{https://oeis.org}.

\bibitem{SD} Sergio Falcón Santana and José Luis Díaz-Barrero, Some properties of sums involving Pell numbers, \textit{Missouri Journal of Mathematical Sciences} 18.1 (2006): 33--40.

\end{thebibliography}
\end{document}